\documentclass[twoside, 11pt]{amsart}
\usepackage{amsmath, amssymb, amsthm}

\title[Flux of strictly contact isotopies]{A note on the volume flux of smooth and continuous strictly contact isotopies}
\author{Stefan M\"uller}
\email{mueller@kias.re.kr}
\address{Korea Institute for Advanced Study, Seoul 130--722, Republic of Korea}

\newtheorem{thm}{Theorem}
\newtheorem{cor}[thm]{Corollary}
\newtheorem{lem}[thm]{Lemma}
\newtheorem{pro}[thm]{Proposition}
\theoremstyle{definition}
\newtheorem{exa}[thm]{Example}

\def\Flux{{\rm Flux}}
\def\Hom{{\rm Hom}}
\def\L{{\mathcal L}}
\def\R{{\mathbb R}}
\def\vol{{\rm vol}}
\def\Vol{{\rm Vol}}
\def\Z{{\mathbb Z}}

\begin{document}
\thispagestyle{plain}

\begin{abstract}
This note on the flux homomorphism for strictly contact isotopies complements the recent paper \cite{mueller:hvf11} by P.~Spaeth and the author.
We determine the volume flux restricted to symplectic and volume-preserving contact isotopies and their $C^0$-limits for some classes of symplectic and contact manifolds and for a number of examples.
In particular, we see that the restriction of the flux may fail to be surjective.
It vanishes for an isotopy preserving a regular contact form, but can be nontrivial for non-regular contact forms.
Applications are discussed in the article cited above.
We also find obstructions to regularizing a strictly contact isotopy that are not present for Hamiltonian isotopies \cite[Section 5.2]{polterovich:ggs01} or contact isotopies \cite{mueller:gcd11}.
\end{abstract}

\maketitle

\section{Introduction} \label{sec:intro}
Let $M$ denote a smooth manifold equipped with a volume form $\mu$.
For simplicity, assume $M$ is closed and connected.
We will discuss the flux homomorphism (see \cite{banyaga:scd97} and the references therein), which is defined for any volume-preserving isotopy, in the cases the manifold admits a symplectic form (if $M$ is even-dimensional) or a contact form (if $M$ is odd-dimensional), and the volume form in question is the canonical one (up to scaling) induced by the symplectic or contact form.
By passing to an appropriate quotient of the codimension $1$ cohomology group of the underlying manifold, the flux is also defined for the time-one maps, i.e.\ volume-preserving diffeomorphisms isotopic to the identity.
For all our purposes it is sufficient to study the flux homomorphism on the Lie algebras of symplectic and strictly contact vector fields, which is given by
	\[ \Flux (\{ X_t \}_{0 \le t \le 1}) = \left[ \int_0^1 \iota (X_t) \mu \, dt \right] \in H^{\dim M - 1} (M,\R). \]
The flux of a volume-preserving isotopy $\{ \varphi_t \}_{0 \le t \le 1}$ is by definition the flux of its infinitesimal generator $\{ X_t \}_{0 \le t \le 1}$, that is, the smooth family of vector fields uniquely determined by $\frac{d}{dt} \varphi_t = X_t \circ \varphi_t$.

\section{Symplectic manifolds}
In this section, $\omega$ denotes a symplectic form on $M^{2n}$, and the volume form is the induced Liouville volume form $\omega^n$.
There is a symplectic version of the flux homomorphism\footnote{Throughout this note, in the absence of the prefix symplectic, the name flux always refers to the map defined in Section~\ref{sec:intro}.} defined by
	\[ \Flux (\{ X_t \}_{0 \le t \le 1}) = \left[ \int_0^1 \iota (X_t) \omega \, dt \right] \in H^1 (M,\R) \]
for $\{ X_t \}_{0 \le t \le 1}$ a smooth family of symplectic vector fields.
All the results recalled in this section are well-known.

\begin{pro}[\cite{banyaga:sgd78}]
The flux homomorphisms for symplectic and for volume-preserving isotopies are both surjective, and are related (up to multiplication with the constant factor $n$) by the map 
\begin{equation} \label{eqn:omega}
	\wedge [\omega^{n - 1}] \colon H^1 (M,\R) \to H^{2n - 1} (M,\R), \ [\beta] \mapsto [\beta \wedge \omega^{n - 1}].
\end{equation}
In particular, the flux of a Hamiltonian isotopy vanishes, and the image of the flux map restricted to symplectic isotopies coincides with the image of the map $\wedge [\omega^{n - 1}]$.
\end{pro}

We remark that the map $\wedge [\omega^{n - 1}]$ in equation~(\ref{eqn:omega}) coincides up to the (nonsingular) cup product pairing with the pairing 
\begin{equation} \label{eqn:pairing}
	H^1 (M,\R) \times H^1 (M,\R) \to \R, \ ([\alpha], [\beta]) \mapsto \int_M \alpha \wedge \beta \wedge \omega^{n - 1}.
\end{equation}

\begin{exa}
If $H^1 (M,\R) = 0$, the restriction of the flux to symplectic isotopies vanishes identically.
By duality, $H^1 (M,\R)$ and $H^{2n - 1} (M,\R)$ have the same rank over $\R$, so the flux is in fact trivial on the full group of volume-preserving isotopies.
This is the case for example for $M = S^2$, or more generally, any complex projective space.
\end{exa}

\begin{exa}
If the map in equation~(\ref{eqn:omega}) is an isomorphism, $(M,\omega)$ is said to be of Lefschetz type.
This class contains all K\"ahler manifolds, such as tori, complex projective spaces, and surfaces.
Tori in all dimensions and surfaces of genus greater than zero of course have nontrivial first cohomology groups.
\end{exa}

\begin{exa}
If $H^1 (M,\R)$ has rank $1$, the pairing in equation~(\ref{eqn:pairing}) vanishes for degree reasons, and thus the map in equation~(\ref{eqn:omega}) is trivial.
However, as remarked above, $H^{2n - 1} (M,\R)$ is nontrivial in this case, and since the flux is surjective, the image of the flux restricted to symplectic isotopies is properly smaller than the image of the flux on the full group of volume-preserving isotopies.
\end{exa}

\section{Contact manifolds}
In this section, $\alpha$ denotes a contact form on a contact manifold $(M^{2n + 1},\xi)$, and the induced volume form is $\alpha \wedge (d\alpha)^n$.
An isotopy $\{ \varphi_t \}_{0 \le t \le 1}$ is said to be contact if it preserves the contact structure $\xi$, or equivalently, its infinitesimal generator $\{ X_t \}_{0 \le t \le 1}$ is contact, that is, it satisfies $\L_{X_t} \alpha = h_t \alpha$ for all $0 \le t \le 1$, where $h_t$ is a smooth function on $M$.
To every contact isotopy corresponds a unique smooth function $H \colon [0,1] \times M \to \R$, defined by $\alpha (X_t) = H_t$, and conversely, every (time-dependent) smooth function $H$ defines a unique contact isotopy.
This function is often called the contact Hamiltonian of the isotopy $\{ \varphi_t \}$.
If $\{ X_t \}$ is in addition divergence-free, then $h_t = 0$, that is, $X_t$ preserves the contact form $\alpha$.
Equivalently, $\{ \varphi_t \}$ is also volume-preserving, and thus preserves the contact form $\alpha$, i.e.\ $\varphi_t^* \alpha = \alpha$ for all $0 \le t \le 1$.
In that case, both $X_t$ and $\varphi_t$ are called strictly contact.
This notion depends on the contact form $\alpha$ and not just on the contact structure $\xi = \ker \alpha$.
The unique strictly contact vector field $R_\alpha$ corresponding to the constant function $1$ is called the Reeb vector field, and its flow the Reeb flow.
The isotopy $\{ \varphi_t \}$ is strictly contact if and only if its Hamiltonian is invariant under the Reeb flow.
These functions are also known as basic functions.

\begin{pro} \label{pro:flux}
If $X_H = \{ X_H^t \}$ is a smooth family of strictly contact vector fields (or equivalently, $H$ is a time-dependent basic function) on $M$, then
\begin{equation} \label{eqn:flux}
	\Flux (X_H) = (n + 1) \left[ \int_0^1 H_t \, dt \cdot (d\alpha)^n \right].
\end{equation}
\end{pro}

This means in particular the image of the flux has a single `generator' $(d\alpha)^n$ over the algebra $C_b^\infty (M)$ of basic functions on $M$.
Note $H^{2n} (M,\R)$ can be quite `large': if $M$ is a closed, orientable $3$-manifold, and $\Sigma = \Sigma_g$ is a closed, orientable surface of genus $g$, then there exists a contact structure on $M \times \Sigma$, and $H^4 (M \times \Sigma)$ has rank at least $2g$.

\begin{proof}
A straightforward computation yields 
	\[ \iota (X_H^t) \left( \alpha \wedge (d\alpha)^n \right) = H_t (d\alpha)^n - n \, \alpha \wedge \left( \iota (X_H^t) d\alpha \right) \wedge (d\alpha)^{n - 1}. \]
Note that since $X_H^t$ is divergence-free, the left-hand side is closed by Cartan's formula.
Since $X_H^t$ is strictly contact,
\begin{equation} \label{eqn:cartan}
	0 = \L_{X_H^t} \alpha = d \left( \iota (X_H^t) \alpha \right) + \iota (X_H^t) d\alpha = dH_t + \iota (X_H^t) d\alpha.
\end{equation}
Combining the above, we obtain 
	\[ \iota (X_H^t) \left( \alpha \wedge (d\alpha)^n \right) = \left( H_t \, d\alpha + n \, \alpha \wedge dH_t \right) \wedge (d\alpha)^{n - 1}. \]
The equality 
	\[ d (H_t \alpha) = dH_t \wedge \alpha + H_t \, d\alpha = H_t \, d\alpha - \alpha \wedge dH_t \]
shows that $H_t \, d\alpha + n \, \alpha \wedge dH_t$ coincides with $(n + 1) H_t \, d\alpha$ up to an exact form, 
proving the proposition.
\end{proof}

In particular, we showed that $H_t (d\alpha)^n$ is a closed form (and thus so is $\alpha \wedge dH_t \wedge (d\alpha)^{n - 1}$).
This can also be seen directly:
	\[ d (H_t (d\alpha)^n) = - (\iota (X_H^t) d\alpha) \wedge (d\alpha)^n = - \frac{1}{n + 1} \iota (X_H^t) (d\alpha)^{n + 1} \]
which vanishes for dimension reasons (we have used equation~(\ref{eqn:cartan}) here), and
	\[ d (\alpha \wedge dH_t \wedge (d\alpha)^{n - 1}) = dH_t \wedge (d\alpha)^n = d (H_t (d\alpha)^n) = 0. \]

\begin{cor} \label{cor:reeb}
The flux of the Reeb flow vanishes.
\end{cor}

\begin{proof}
By Proposition~\ref{pro:flux}, $\Flux (R_\alpha) = (n + 1) [(d\alpha)^n] = 0$.
\end{proof}

\begin{cor} \label{cor:regular}
If $(M,\alpha)$ is regular (i.e.\ the Reeb vector field determines a free $S^1$-action), the flux of any strictly contact isotopy vanishes identically.
\end{cor}

\begin{proof}
Denote by $(B,\omega)$ the quotient of $(M,\alpha)$ by the Reeb flow, and by $p \colon M \to B$ the natural projection.
The map $p^* \colon C^\infty (B) \to C_b^\infty (M,\alpha)$ is an (algebra) isomorphism, so $H_t (d\alpha)^n = p^* (F_t \, \omega^n)$ for a smooth family of not necessarily mean value zero normalized Hamiltonians $F_t$ on $B$.
See \cite{boothby:cm58, banyaga:gdp78}.
Then 
	\[ \int_B (F_t - c_t) \, \omega^n = 0, \mbox{ where } c_t = \frac{1}{\int_B \omega^n} \int_B F_t \, \omega^n \]
is the mean value of $F_t$ on $B$ with respect to the volume form $\omega^n$.
That means the $2n$-form $(F_t - c_t) \, \omega^n = d\gamma_t$ is exact, and therefore
\begin{equation} \label{eqn:exactness}
	H_t (d\alpha)^n = (H_t - c_t) (d\alpha)^n + c_t (d\alpha)^n = d ( p^*\gamma_t + c_t \, \alpha \wedge (d\alpha)^{n - 1} )
\end{equation}
is an exact form on $M$.
The claim now follows from Proposition~\ref{pro:flux}.
\end{proof}

The identity in equation (\ref{eqn:exactness}) is used in \cite{mueller:hvf11} to calculate an explicit formula for the helicity of a strictly contact vector field on a regular contact $3$-manifold.

The pair $(M,\alpha)$ is regular if and only if it is the prequantization bundle of a (necessarily integral) symplectic manifold $(B,\omega)$.
From this it is easy to construct contact manifolds $(M,\xi = \ker \alpha)$ where the flux (when restricted to strictly contact isotopies with respect to $\alpha$) is not surjective.

\begin{exa} \label{exa:nontrivial}
Choose a closed and connected manifold $B^{2n}$ with integral symplectic form $\omega$, whose fundamental group is not perfect.
For example, any torus $T^{2n}$ with its standard symplectic form.
Let $(M^{2n + 1},\alpha)$ be its prequantization bundle.
From the long exact sequence on homotopy of the bundle $S^1 \hookrightarrow M \stackrel{p}{\rightarrow} B$, we see the homomorphism $p_* \colon \pi_1 (M) \to \pi_1 (B)$ is surjective, and in particular, $\pi_1 (M)$ is nontrivial.
By hypothesis, the abelianization $H_1 (M,\R)$ of $\pi_1 (M)$ is nontrivial.
Then by Poincar\'e duality, $H^{2n} (M,\R) \cong H_1 (M,\R) \not= 0$.
But by Corollary~\ref{cor:regular}, the flux is trivial on strictly contact isotopies.
\end{exa}

For strictly contact $S^1$-actions, we note the following (which gives another proof of Corollary~\ref{cor:reeb} in those cases where the Weinstein conjecture holds).

\begin{pro}
If the vector field $X$ induces a strictly contact $S^1$-action on $(M,\alpha)$ that has at least one contractible orbit, then the induced strictly contact isotopy has vanishing flux.
\end{pro}

\begin{proof}
The orbits of the $S^1$-action are all homologous, and since there exists a contractible orbit, the represented homology class $\beta_X$ is zero.
It is easy to see that $\beta_X$ is Poincar\'e dual to the cohomology class of $\iota (X) (\alpha \wedge (d\alpha)^n)$, and therefore the flux vanishes.
\end{proof}

\begin{exa} \label{exa:3-torus}
Let $M = T^3$ be the $3$-torus with contact form 
	\[ \alpha = \cos z \, dx - \sin z \, dy, \]
where $x$, $y$, $z \in \R / (2 \pi \Z)$ are the coordinates on $T^3$.
We compute 
	\[ d\alpha = \sin z \, dx \wedge dz + \cos z \, dy \wedge dz, \]
so that the Reeb vector field is given by 
	\[ R_\alpha = \cos z \, \frac{\partial}{\partial x} - \sin z \, \frac{\partial}{\partial y}. \]
In particular, any function on $M$ that depends only on $z$ is basic.
The induced volume form $\alpha \wedge d\alpha = dx \wedge dy \wedge dz$ is the standard volume form on $T^3$.
Choosing $H = \sin z \, / (2 \pi)^2$ and $H = \cos z \, / (2 \pi)^2$, Proposition~\ref{pro:flux} gives $\iota (X_H) = dx \wedge dz$ and $dy \wedge dz$ respectively.
By equation (\ref{eqn:flux}), the $2$-form $dx \wedge dy$ does not lie in the image of the flux when restricted to strictly contact isotopies.

Another way of seeing this is the following lemma, which shows that a basic function on $(T^3,\alpha)$ is always independent of $x$ and $y$.

\begin{lem} \label{lem:dense-orbits}
In the situation above, any basic function $H_t$ is independent of $x$ and $y$.
In particular, $C_b^\infty (T^3,\alpha) \cong C^\infty (S^1)$.
\end{lem}

\begin{proof}
For fixed $x_0$, $y_0 \in S^1$, denote by $\Delta$ the function 
	\[ \Delta = \Delta_{x_0,y_0} \colon z \mapsto \left( \frac{\partial H_t}{\partial x}, \frac{\partial H_t}{\partial y} \right) (x_0, y_0, z). \]
We note when the `slope' $-\tan z_0$ is irrational, the Reeb orbits are dense in the $z = z_0$-`plane', and thus the map $(x,y) \mapsto H_t (x,y,z_0)$ is constant.
Consequently, for any choice of $x_0$ and $y_0$ above, the function $\Delta$ vanishes at all `angles' $z_0$ with irrational $\tan z_0$.
By continuity, it must be identically zero, proving the claim.

Alternatively, fix $z \in S^1$ and consider the Fourier series of
	\[ H_z (x,y) = H (x,y,z) = \sum_{j,k} h_{j,k} e^{i (j x + k y)}. \]
If $H$ is basic, then $R_\alpha . H_z (x,y) = R_\alpha . H (x,y,z) = 0$, or
	\[ \cos z \, \frac{\partial H_z}{\partial x} = \sin z \, \frac{\partial H_z}{\partial y}, \]
which in terms of the Fourier coefficients is equivalent to $j \cos z = k \sin z$ for all $j, k$ with $h_{j,k} \not= 0$.
If $\tan z$ is irrational, $h_{j,k}$ must vanish for all $j,k$ except possibly the constant term $h_{0,0}$.
In other words, the partial derivatives of $H$ in the direction of $x$ and $y$ vanish at all points $(x,y,z)$ with $\tan z$ irrational.
By the same argument as above, $H$ is independent of $x$ and $y$.
\end{proof}
\end{exa}

By the previous lemma, given two linearly independent cotangent vectors at a point $y \in T^3$, it is not possible to construct a basic function $G$ with partial derivatives in a prescribed direction equal to the given cotangent vectors.
This is where L.~Polterovich's argument on regularizing a Hamiltonian isotopy \cite[Section 5.2]{polterovich:ggs01} breaks down in the strictly contact case.
See \cite{mueller:hvf11} for details and a proof in the contact case.

\begin{exa}
Let $M = T^{n + 1} \times S^n$, with contact form 
	\[ \alpha = \sum_{k = 0}^n y_k \, dx_k, \]
where $x_0, \ldots, x_n \in \R / \Z$ are coordinates on $T^{n + 1}$, and $y_0, \ldots, y_n$ are coordinates on the unit sphere $S^n \subset \R^{n + 1}$.
Clearly $d\alpha = - \sum dx_k \wedge dy_k$, and
	\[ (d\alpha)^n = (-1)^{\frac{n (n + 1)}{2}} n! \sum_{k = 0}^n dx_0 \wedge \ldots \widehat{dx_k} \ldots \wedge dx_n \wedge dy_0 \wedge \ldots \widehat{dy_k} \ldots \wedge dy_n, \]
where $\widehat{dx_k}$ (or $\widehat{dy_k}$) means that index is omitted.
The Reeb vector field is 
	\[ R_\alpha = \sum_{k = 0}^n y_k \frac{\partial}{\partial x_k}, \]
and thus the functions 
\begin{equation} \label{eqn:h_k}
	H_k = \frac{(-1)^{\frac{n (n + 1)}{2} + k}}{c_n \cdot n!} \cdot y_k
\end{equation}
are basic, where $c_n = \vol (S^n)$ denotes the volume of the unit $n$-sphere with respect to the standard volume form 
	\[ d\Vol = d\Vol (S^n) = \sum_{k = 0}^n (-1)^k y_k \, dy_0 \wedge \ldots \widehat{dy_k} \ldots \wedge dy_n. \]
A similar argument as in Lemma~\ref{lem:dense-orbits} proves $C_b^\infty (T^{n + 1} \times S^n,\alpha) \cong C^\infty (S^n)$.
Assume for now that $n > 1$.
The standard basis $\{ a_k \}$ of $H_{2n} (M,\R)$ can be represented by the embedded submanifolds $\{ x_k = 0 \}$.
A direct computation shows
	\[ \int_{a_j} (n + 1) H_k (d\alpha)^n = \delta_{j k}, \]
i.e.\ $\{ (n + 1) [H_k (d\alpha)^n] \}$ forms a basis of $H^{2n} (M,\R)$ dual to the standard basis $\{ a_k \}$ of $H_{2n} (M,\R)$.
That shows the flux is surjective even when restricted to strictly contact isotopies.
In fact, the map
	\[ dx_0 \wedge \ldots \widehat{dx_k} \ldots \wedge dx_n \wedge d\Vol \mapsto \Phi_{H_k} \]
defines a continuous homomorphic section of the flux.
If $n = 1$, the contact form $\alpha$ is diffeomorphic to the one in Example~\ref{exa:3-torus}.
\end{exa}

Recall the flux can be interpreted as an obstruction to fragmentation of a vector field \cite[page 15]{banyaga:scd97}.
Strictly contact vector fields do not posses the fragmentation property in general.

\section{Continuity of the flux homomorphism}

Let $\mu$ denote a volume form on $M$, normalized so that $\int_M \mu = 1$.
Recall the usual identification of $H_1 (M,\R)$ with $\Hom([M,S^1],\R)$, via identification of $H^1 (M,\Z)$ with $[M,S^1]$, applying the $\Hom (-,\Z)$ functor, and then taking the tensor product with $\R$.

\begin{pro}[\cite{fathi:sgh80}] \label{pro:poincare}
The flux homomorphism $\Flux$ is Poincar\'e dual to the mass flow homomorphism $\Theta$, and the latter is $C^0$-continuous.
More precisely, let $\sigma$ denote the canonical volume form on $S^1$ given by the natural orientation of the circle.
Then for any $f \colon M \to S^1$,
	\[ \int_M \Flux (X) \wedge f^* \sigma = \Theta (X) (f). \]
\end{pro}

As a consequence of this proposition, we can define the flux or mass flow of a topological or continuous Hamiltonian, symplectic, strongly symplectic, or strictly contact isotopy \cite{mueller:ghh07, mueller:ghc08, banyaga:gss10, banyaga:ctg11} by extending continuously.
The proof of the following lemma, which completely determines the image of the above continuous extensions of the flux (or mass flow) homomorphism, is straightforward.

\begin{lem}
If a vector field $X$ is divergence-free (symplectic, strictly contact), then $\lambda \cdot X$ is divergence-free (symplectic, strictly contact) for any $\lambda \in \R$, and $\Flux (\lambda X) = \lambda \cdot \Flux (X)$.
The flux homomorphism (as well as the restriction to symplectic or strictly contact vector fields) is a homomorphism of vector spaces.
Thus the image is a linear subspaces of $H^k (M,\R)$, where $k = \dim M - 1$, and in particular is closed.
\end{lem}

A similar remark applies when passing to the groups of time-one maps.
For symplectic vector fields, the last part of the lemma also follows from equation~(\ref{eqn:omega}), and for strictly contact vector fields from Proposition~\ref{pro:flux}.
By composing with the inverse of the (vector space) isomorphism that assigns to a strictly contact vector field $X$ the basic function $\alpha (X)$, the flux homomorphism can also be viewed as a (vector space) homomorphism $C_b^\infty (M) \to H^k (M,\R)$.
Under the above identification of $H_1 (M,\R)$ with $\Hom([M,S^1],\R)$, we have the following explicit formula for the mass flow of a strictly contact isotopy.

\begin{pro}
The mass flow of a strictly contact isotopy generated by a basic contact Hamiltonian $H \colon [0,1] \times M \to \R$, and evaluated on a function $f \colon M \to S^1$, is given by the formula
	\[ (n + 1) \int_M \left( \int_0^1 H_t \, dt \cdot (R_\alpha . f) \right) \mu, \]
where $\mu = \alpha \wedge (d\alpha)^n$ is the canonical volume form on $M^{2 n + 1}$ induced by $\alpha$.
\end{pro}

\begin{proof}
Combine Proposition~\ref{pro:flux} and Proposition~\ref{pro:poincare}.
\end{proof}

\section*{Acknowledgments}
The first part of Example~\ref{exa:3-torus} (the image of the flux has rank at least two) was worked out jointly with A.~Banyaga and Spaeth during a visit of Banyaga to Korea Institute for Advanced Study, and was the starting point of this note.
I would like to thank Banyaga and Spaeth for stimulating discussions, Y.~Eliashberg for suggesting to me the construction in Example~\ref{exa:nontrivial} during a visit at MSRI, and KIAS for inviting Banyaga to Seoul and financial support for my visit to MSRI.

\bibliography{flux}
\bibliographystyle{amsalpha}
\end{document}